\documentclass[12pt]{amsart}
 \usepackage[left=3cm, right=3cm, top=4cm, bottom=4cm]{geometry} 
\usepackage[utf8]{inputenc}
\usepackage{amssymb}
\usepackage{amscd}
\usepackage{amsmath}
\usepackage{enumerate}
\usepackage{comment}
\usepackage{dsfont}
\usepackage{amsthm}
\usepackage{url}
\usepackage{ mathrsfs }
\usepackage{lineno}
\usepackage{textcomp}
\usepackage[all]{xy}
\usepackage{tikz-cd}
\usepackage{xcolor}
\definecolor{darkgrey}{rgb}{0.4,0.4,0.5}
\usepackage{hyperref}
\hypersetup{
    colorlinks=true,
    linkcolor=blue,
    citecolor=darkgrey
}

\usepackage{url}

\usepackage{multirow}

\newtheorem{theorem}{Theorem}[section]
\newtheorem{lemma}[theorem]{Lemma}

\newtheorem{conjecture}[theorem]{Conjecture}
\newtheorem{corollary}[theorem]{Corollary}
\newtheorem{proposition}[theorem]{Proposition}

\newtheorem{proposition-definition}[theorem]{Proposition-Definition}

\newtheorem{remark}[theorem]{Remark}

\numberwithin{equation}{section}

\newcommand\EatDot[1]{}

\newcommand{\cyc}{{\mathrm{cyc}}}

\newcommand{\cH}{{\mathcal{H}}}

\newcommand{\Gal}{{\mathrm{Gal}}}
\newcommand{\Ker}{{\mathrm{Ker}}}
\newcommand{\Coker}{{\mathrm{Coker}}}
\newcommand{\Sel}{{\mathrm{Sel}}}

\newcommand{\Q}{{\mathbb Q}}
\newcommand{\Z}{{\mathbb Z}}

\newcommand{\Sss}{S^\mathrm{ss}}
\newcommand{\Trace}{\mathrm{Trace}}

\newcommand{\Hyptwo}{\mathrm{HYP}_1}

\newcommand{\WLCF}{\mathrm{(WLC)}_F}
\newcommand{\cG}{\mathcal{G}}

\begin{document}
\title[Iwasawa $\mu$-invariant for supersingular curves]{On the growth of $\mu$-invariant in Iwasawa theory of supersingular Elliptic curves}

\author{Jishnu Ray}
\address{Department of Mathematics, The University of British Columbia}
\curraddr{Room 121, 1984 Mathematics Road\\
	Vancouver, BC\\
	Canada V6T 1Z2}
\email{jishnuray1992@gmail.com}

\thanks{This work is supported by PIMS-CNRS postdoctoral research grant from the University of British Columbia.}

\begin{abstract}
In this article, we provide a relation between the $\mu$-invariants of the dual plus and minus Selmer groups for supersingular elliptic curves when we ascend from the cyclotomic $\Z_p$-extension  to a $\Z_p^2$-extension over an imaginary quadratic field. 
Furthermore we show that the (supersingular) $\mathfrak{M}_H(G)$-conjecture is  \textit{equivalent} to the fact that the $\mu$-invariant doesn't change as we go up the tower.

\end{abstract}

\maketitle

%\tableofcontents


\section{Introduction}
Iwasawa theory of elliptic curves which have good ordinary reduction at a prime $p$ over the trivializing extension was studied in \cite{Coates_Fragments}, \cite{CoatesSchneiderSujatha_Links_between}. The main ingredient in these papers is the fact that when the elliptic curve over $\Q$ has good ordinary reduction at $p$, then the dual Selmer group over the cyclotomic extension is a torsion module over the corresponding Iwasawa algebra. Then, in \cite[Prop. 2.12]{CoatesSchneiderSujatha_Links_between}, the authors study a relation between $\mu$-invariants of the dual Selmer groups over the cyclotomic extension and the dual Selmer group over the (non-abelian) trivializing extension. Under this setting, the five authors in \cite{CFKSV} framed  the well know $\mathfrak{M}_H(G)$-conjecture in Iwasawa theory. In \cite[Prop. 2.12]{CoatesSchneiderSujatha_Links_between}, the authors also showed that under $\mathfrak{M}_H(G)$-conjecture, the $\mu$-invariant of the dual Selmer group over the cyclotomic extension is the same as the $\mu$-invariant of the dual Selmer group over the trivializing extension.

 In this paper one of our main result is the relation between $\mu$-invariants of dual plus and minus Selmer groups as we ascend from a  $\Z_p$-extension  to  an abelian $\Z_p^2$-tower, when $E$ has supersingular reduction at $p$  over an imaginary quadratic field. The reason of working over imaginary quadratic fields is to ensure that the dual plus and minus Selmer groups are already known and proved to be torsion modules and weak Leopoldt's conjecture is true. Further, in this case of imaginary quadratic field, another main result in this article is to show that (supersingular) $\mathfrak{M}_H(G)$-conjecture is \textit{equivalent} to the fact that the $\mu$-invariants of the dual plus and minus  Selmer groups over the cyclotomic $\Z_p$-extension are the same as the   $\mu$-invariants of the dual plus and minus  Selmer groups over the  $\Z_p^2$-extension. Note that in the ordinary case setting in \cite[Prop. 2.12]{CoatesSchneiderSujatha_Links_between}, the $p$-adic Lie group is non commutative and has higher cohomological dimension and so the authors could show only one direction (the fact that $\mathfrak{M}_H(G)$-conjecture implies the $\mu$-invariant doesn't change). However, in our abelian setting over imaginary quadratic fields and supersingular elliptic curves, we could show the converse direction. As a corollary of our main result linking the $\mu$-invariants, even \textit{without} assuming $\mathfrak{M}_H(G)$-conjecture, we can show that if the $\mu$-invariants of the dual plus and minus  Selmer groups over the cyclotomic $\Z_p$-extension are zero, then  the  $\mu$-invariants of the dual plus and minus  Selmer groups over the  $\Z_p^2$-extension are also zero. 
 Please see Theorem \ref{Thm:all} where we summarize all our results.
  Under Iwasawa main conjecture, our results can also be translated in terms of the (analytic) $\mu$-invariants of the cyclotomic plus and minus $p$-adic $L$-functions and the two variable plus and minus $p$-adic $L$-functions of Loeffler and Zerbes which correspond to the plus and minus Selmer groups over corresponding Iwasawa tower.
  
  We are aware that it is possible to possible to generalize our results in this paper to any number field over a $\Z_p^d$-tower ($d>2$)  but then the formula for the $\mu$-invariants would not be so simple and one has to \textit{assume} that the dual signed Selmer groups are torsion modules over the corresponding Iwasawa algebra  and weak Leopoldt's conjecture over the $\Z_p^d$-extension is true.  Further, in this more general case, the direction that $\mathfrak{M}_H(G)$-conjecture is true implies that the $\mu$-invariants doesn't change, is not clear to me. Also the existence of Euler systems are unknown for general extensions, whereas in the imaginary quadratic field case they are known \cite{LLZ1}. Therefore, in this paper we only treat the case $d=2$ where the results are as nicest looking as possible.

\subsubsection{Preliminaries and our main results}
Let $F^{\prime}$ be a subfield of $F$, $E/F^{\prime}$ be an elliptic curve with  good supersingular reduction at all primes above $p$. Let $S$ be the set of primes of $F^\prime$ above $p$ and the primes where $E$ has bad reduction. 
We write $\Sss_p$ to be the set of primes of $F^\prime$ lying above $p$ where $E$ has supersingular reduction. Then, $S=\Sss_p \cup \text{ \{bad primes\}}$. Let $F^{\prime}_v$  be the completion of $F^{\prime}$ at a prime $v \in \Sss_p$ with residue field $k^{\prime}_v$. Let $\tilde{E}(k^{\prime}_v) $ be the $k^{\prime}_v$-points of the reduction of $E$ at place $v$. 

Assume the following.
\begin{enumerate}[(i)]
	\item $\Sss \neq 0$,
	\item For all $ v \in \Sss$, the completion of $F^{\prime}_v$ at $v$, denoted $F^{\prime}_v$ is $\Q_p$,
	\item $a_v = 1+p -\#\tilde{E}(k^{\prime}_v) =0,$
	\item $v$ is unramified in $F$.
\end{enumerate}
Recall that $F_{\cyc}$ is the cyclotomic $\Z_p$-extension of $F$. For each integer $n\geq 0$, let $F_n$ be the sub-extension of $F_{\cyc}$ such that $F_n$ is a cyclic extension of degree $p^n$ over $F$. Let $\Sss_{p,F}$ denote the set of primes of $F$ above $\Sss$. By abuse of notation, let $F_{n,v}$ be the completion of $F_n$ at the unique prime over $v \in \Sss_{p,F}$. For every $v \in \Sss_{p,F}$, following Kobayashi \cite{Kobayashi}, we define
\begin{align}
E^+(F_{n,v}) &= \{ P \in \hat{E}(F_{n,v}) \mid  \Trace_{n/m+1}P \in  \hat{E}(F_{m,v}) \text{ for all even } m, 0 \leq m \leq n-1 \},\label{eq:traceEven}\\
E^-(F_{n,v}) &= \{ P \in \hat{E}(F_{n,v}) \mid  \Trace_{n/m+1}P \in  \hat{E}(F_{m,v}) \text{ for all odd } m, 0 \leq m \leq n-1 \}\label{eq:traceOdd},
\end{align}
Here $ \Trace_{n/m+1}$ is the trace map from   $\hat{E}(F_{n,v})$ to $\hat{E}(F_{m+1,v})$. 
We define the  ($p$-adic) Selmer group over $F_n$ by the following sequence.
$$0 \rightarrow \Sel_p(E/F_n) \rightarrow  \ H^1(F_S/F_n, E_{p^\infty}) \rightarrow  \oplus_{w \mid S}H^1(F_{n,w},E)(p).  $$
Here $F_S$ is the maximal extension of $F$ unramified outside primes of $F$ over $S$.  The plus and minus Selmer group over $F_n$ are defined by 
$$\Sel_p^{\pm}(E/F_n) = \ker \Big( \Sel_p(E/F_n) \rightarrow \oplus_{v \in \Sss_{p,F}} \frac{H^1(F_{n,v},E_{p^\infty})}{E^\pm(F_{n,v}) \otimes \Q_p/\Z_p}\Big). $$
We regard $E^\pm(F_{n,v}) \otimes \Q_p/\Z_p$ as a subgroup of $H^1(F_{n,v},E_{p^\infty})$ via the Kummer map. 
The plus and minus Selmer groups over $F_{\cyc}$ are defined by $$\Sel_p^\pm(E/F_\cyc) = \varinjlim_n \Sel_p^{\pm}(E/F_n).$$

\subsubsection{Signed Selmer groups over a $\Z_p^2$-extension}\label{inftySuper}
% and $F^{\prime} = \Q$ (therefore $\#\Sss_{p,F}=2$). Then generalizing work of Kobayashi, Kim showed that it is possible to define signed Selmer groups over a $\Z_p^2$-extension \cite{Kim}. 
Suppose $F$ is an imaginary quadratic field where $p$ splits completely.
Let $F_\infty$ denote the compositum  of all $\Z_p$-extensions of $F$. 

%(Here $F_\infty$ is \textit{not }$F(E_{p^\infty})$  although we have used the same subscript $\infty$ as in   Section \ref{notationCoates}). 

By Leopoldt's conjecture we know that $G= \Gal(F_\infty/F)\cong \Z_p^2$, which implies that $F_\infty$ over $F_\cyc$ is a $\Z_p$-extension. Let $v$ be a place of $F$ and $w$ be a place of $F_\infty$ above $v$. If $v \mid p$, then $F_v \cong \Q_p$ and $F_{\infty, w}$ is an abelian pro-$p$ extension over $F_v$. By local class field theory, $\Gal(F_{\infty,w}/F_v) \cong \Z_p^2.$ 
%We can identify $F_{\infty,w}$ with $F_{v,\cyc}^{\mathrm{ur}}$ which is the compositum of the cyclotomic $\Z_p$-extension of $F_v$ and the unramified $\Z_p$-extension of $F_v$. 
Under this setting,  it is possible to define the plus and minus norm groups $E^{\pm}(F_{\infty,w}) \subset \hat{E}(F_{\infty,w})$ via Trace maps as in \eqref{eq:traceEven} and \eqref{eq:traceOdd}  (cf. see Section 5.2 of \cite{LeiSujatha} which is a generalization of a construction by Kim \cite{Kim}). Identifying $E^{\pm}(F_{\infty,w})  \otimes \Q_p/\Z_p$ with a subgroup of $H^1(F_{\infty,w},E_{p^\infty})$ via the Kummer map, we may define the local terms 
$$J_v^\pm(E/F_\infty) = \oplus_{w \mid v} \frac{H^1(F_{\infty,w},E_{p^\infty})}{E^{\pm}(F_{\infty,w})  \otimes \Q_p/\Z_p}.$$
Then, the plus and minus Selmer groups over $\Z_p^2$-extension $F_\infty$ are defined by 
\begin{equation}
\Sel_p^{\pm}(E/F_\infty) = \ker \Big( \Sel_p(E/F_\infty)  \rightarrow \oplus _{v \in \Sss_{p,F}}J_v^\pm(E/F_\infty) \Big).
\end{equation}
Here the (classical) Selmer group $\Sel_p(E/F_\infty)$ is defined by taking inductive limit of Selmer groups over all finite extensions contained in the $p$-adic Lie extension $F_\infty$.

\subsubsection{\textbf{Main result}}
Let $X_\infty^\pm$ denote the Pontryagin dual of $\Sel_p^\pm(E/F_\infty)$ and $X_\cyc^\pm$ denote the Pontryagin  dual of $\Sel_p^\pm(E/F_\cyc)$.  Let $G =\Gal(F_\infty/F)$,  $H=\Gal(F_\infty/F_\cyc)$ and $\Gamma=\Gal(F_\cyc/F)$. The groups $X_\infty^\pm$ and $X_\cyc^\pm$ are finitely generated torsion modules over the Iwasawa algebra of $G$ and $\Gamma$ respectively. Let $X_\infty^\pm(p)$ be the submodule of all the elements of $X_\infty^\pm$ which are annihilated by some power of $p$ and $X_{\infty,f}^\pm=X_\infty^\pm/X_{\infty}^\pm(p)$.

Then we show that
\begin{theorem}[See Theorems \ref{mainTheoremSupersingular} and \ref{thm:mufixed}]\label{Thm:all}
We have the relation 	$$\mu_G(X_\infty^\pm) = \mu_\Gamma(X_\cyc^\pm) - \mu_\Gamma(H_0(H, X_{\infty,f}^\pm)).$$
	
	So if $ \mu_\Gamma(X_\cyc^\pm) =0$ then $\mu_G(X_\infty^\pm)=0$. 
	Furthermore, $\mu_G(X_\infty^\pm) = \mu_\Gamma(X_\cyc^\pm)$ if and only if  $\mathfrak{M}_H(G)$-conjecture is true (that is $X_{\infty,f}^\pm$ is finitely generated as a $\Lambda(H)$-module).
\end{theorem}
Here note that in the case of Coates-Schneider-Sujatha  \cite[Prop. 2.12]{CoatesSchneiderSujatha_Links_between}, they are working over a noncommutative $p$-adic Lie group of cohomological dimension four and so their $\mu$-invariant formula is not so simple as ours and they could only prove that if $\mathfrak{M}_H(G)$-conjecture is true then the $\mu$-invariant doesn't change. However, in our case $G$ is a commutative $p$-adic Lie group of cohomological dimension $2$ and this allows us to show also the converse. That is, if $\mathfrak{M}_H(G)$-conjecture is true, then the $\mu$-invariants remain fixed.

 In Section \ref{sec:hoschild1}, we recall basic facts on Hochschild-Serre spectral sequence which we will use throughout the article. The proofs of the main results are included in Section \ref{main}.

\section{Recall on Hochschild-Serre spectral sequence}\label{sec:hoschild1}
Let $G$ be any profinite group. Let $M$ be an abelian group with  discrete topology and a continuous action of $G$, $H$ be a closed normal subgroup of $G$, then there exists the following Hochschild-Serre spectral sequence.
\begin{equation}\label{eq:hochserre}
H^r(G/H,H^s(H,M)) \implies H^{r+s}(G,M).
\end{equation}
This gives rise to the following inflation-restriction exact sequence
$$0 \rightarrow H^1(G/H,M^H) \rightarrow H^1(G,M) \rightarrow H^1(H,M)^{G/H} \rightarrow H^2(G/H,M^H) \rightarrow H^2(G,M).$$
Furthermore, if $H^i(H,M)=0$ for $2 \leq i \leq m$, then we get the following exact sequence for higher cohomology groups.
\begin{align*}
H^m(G/H,M^H) &\rightarrow H^m(G,M) \rightarrow H^{m-1}(G/H,H^1(H,M)) \\
& \rightarrow H^{m+1}(G/H,M^H) \rightarrow H^{m+1}(G,M).
\end{align*}
We are going to use the above two exact sequences  throughout this article.

\section{Relation of $\mu$-invariants between cyclotomic and $\Z_p^2$-extension}\label{main}
Recall that $H = \Gal(F_\infty/F_\cyc)$, a $p$-adic Lie group of dimension $1$, $\Gamma = \Gal(F_\cyc/F)$ also of dimension $1$ and  $\Lambda(\Gamma)$ is the Iwasawa algebra of $\Gamma$. This is a completed group algebra defined by 
 $$\Lambda(\Gamma)= \varprojlim_W\Z_p[\Gamma/W],$$ where $W$ runs over all open normal subgroups of $\Gamma$. The Iwasawa algebra  $\Lambda(\Gamma)$ is a local Noetherian integral domain. The completed group algebra $\Lambda(\Gamma)$ can also be thought of as a distribution algebra of $\Z_p$-valued measures on the $p$-adic Lie group $\Gamma$.
Let $G_S(F_\infty) = \Gal(F_S/F_\infty)$ and $G_S(F_\cyc) = \Gal(F_S/F_\cyc)$. Consider the following fundamental diagram.
\begin{equation}\label{fundamental_diagram_Supersingular}
\begin{tikzcd}
0 \arrow[r] & \Sel_p^\pm(E/F_\infty)^H  \arrow[r] & H^1(G_S(F_\infty), E_{p^\infty})^H \arrow[r, "\lambda_\infty^{H,\pm} "] & \big(\oplus_{v \in S}J_v^{\pm}(E/F_\infty)\big)^H   \\
0 \arrow[r] & \Sel_p^\pm(E/F_\cyc) \arrow[u, "\alpha_\infty"] \arrow[r] &  H^1(G_S(F_\cyc), E_{p^\infty}) \arrow[u, "\beta_\infty"]  \arrow[r, "\lambda_{\cyc}^\pm"] & \oplus_{v \in S}J_v^\pm(E/F_\cyc) \arrow[u, "\gamma_{\infty}^\pm = \oplus_{v \in S} \gamma_{v,\infty}^\pm"] 
\end{tikzcd}
\end{equation}
The following Lemma is due to Lei and Sujatha.
\begin{lemma}\label{lem1}
	The maps $\alpha_\infty$, $\beta_\infty$ and $\gamma_\infty^\pm$ are isomorphisms.
\end{lemma}
\begin{proof} 
	See Corollary 5.7, Proposition 5.8 and Lemma 5.10 of \cite{LeiSujatha} and use snake lemma for the fundamental diagram \eqref{fundamental_diagram_Supersingular}. Note that we have assumed that $E$ has supersingular reduction for all primes above $p$ and so the case of Lemma 5.9 of \cite{LeiSujatha} does not appear.
\end{proof}

From now on, we will assume that $E$ is \textit{defined over $\Q$} with good supersingular reduction at $p$ and $a_p=0$. 
Then, we have the following Theorems.
\begin{theorem}\label{torsion}
	The Pontryagin dual of $\Sel_p^\pm(E/F_\cyc)$ is $\Lambda(\Gamma)$-torsion. The Pontryagin dual of $\Sel_p^\pm(E/F_\infty)$ is $\Lambda(G)$-torsion.
\end{theorem}
\begin{proof}
	Note that our $\Sel_p^+(E/F_\infty)$ is  $\Sel_p^{++}(E/F_\infty)$ and our $\Sel_p^-(E/F_\infty)$ is   $\Sel_p^{--}(E/F_\infty)$ in the paper \cite{LeiSprung} by Lei and Sprung.
Then this Theorem essentially follows by Lemma \ref{lem1} and from the proof of Theorem 4.4 of \cite{LeiSprung}. 
\end{proof}
\begin{theorem}\label{leopoldt}
	Weak Leopoldt's conjectures for $F_\infty$ and $F_\cyc$ are both true. That is $H^2(F_S/F_\infty, E_{p^\infty}) =0$ and $H^2(F_S/F_\cyc, E_{p^\infty}) =0$.
\end{theorem}
\begin{proof}
	By Theorem \ref{torsion}, the result follows from Proposition 4.6 and Proposition 4.12 of \cite{LeiLim}.
\end{proof}

%Here are some Hypothesis that we will need. 
%\begin{align*}
%\text{(WLC)}_{F_\infty}&: \text{ Weak Leopoldt's conjecture for $F_\infty$: that is } H^2(F_S/F_\infty, E_{p^\infty}) =0. \\
%\text{(WLC)}_{F}&: \text{ Weak Leopoldt's conjecture for $F$: that is } H^2(F_S/F, E_{p^\infty}) =0. \\
%\text{HYP}_1&: \text{The Pontryagin dual of  } \Sel_p^\pm(E/F_\cyc) \text{ is } \Lambda(\Gamma)-\text{torsion}.
%\text{HYP}_2&: \text{The Pontryagin dual of  } \Sel_p^\pm(E/F_\infty) \text{ is } \Lambda(G)-\text{torsion}.
%\end{align*}
\begin{lemma}\label{lemma2}
We have the following short exact sequence.
\begin{equation}\label{Exact1Supersingular}
0 \rightarrow \Sel_p^\pm(E/F_\infty)^H \rightarrow  H^1(G_S(F_\infty), E_{p^\infty})^H \xrightarrow{\lambda_{\infty}^{H,\pm}} \oplus_{v \in S}J_v^{\pm}(E/F_\infty)^H \rightarrow 0.
\end{equation}
That is, the map $\lambda_{\infty}^{H,\pm}$ is surjective.
\end{lemma}
\begin{proof}
By Theorem \ref{torsion}, this is Proposition 5.11 of Lei and Sujatha \cite{LeiSujatha}.
\end{proof}
\begin{lemma}\label{jishnu}
	We have $$H^1(H,H^1(G_S(F_\infty),E_{p^\infty}))=0.$$
\end{lemma}
\begin{proof}
	Note that, as $p$ is odd, $G_S(F)$ has $p$-cohomological dimension $2$. Therefore, as $G_S(F_\infty)$ is a closed subgroup of $G_S(F)$, the $p$-cohomological dimension of $G_S(F_\infty)$ is less than or equal to $2$. This implies that $H^m(G_S(F_\infty),E_{p^\infty}) =0$ for all $m \geq 2$. 
	By Hochschild-Serre spectral sequence we then have $$H^{2}(G_S(F_\cyc),E_{p^\infty}) \rightarrow H^1(H,H^1(G_S(F_\infty),E_{p^\infty})) \rightarrow H^3(H, E_{p^\infty}^{G_S(F_\infty)}).$$
	
	Now by Theorem \ref{leopoldt},  the first term of the above sequence is zero; the third term is also zero as $H$ has $p$-cohomological dimension $1$. Hence, the middle term has to be zero. 
\end{proof}
\begin{proposition}\label{Prop:mithu}
	We have the following exact sequence.
	\begin{equation}\label{exact2}
	0 \rightarrow \Sel_p^\pm(E/F_\infty) \rightarrow  H^1(G_S(F_\infty), E_{p^\infty})  \xrightarrow{\lambda_{\infty}^\pm} \oplus_{v \in S}J_v^{\pm}(E/F_\infty) \rightarrow 0.
	\end{equation}
	That is, the map $\lambda_{\infty}^\pm$  is surjective.
\end{proposition}
\begin{proof}
	Let $A_\infty^\pm$ denote the image of the map $\lambda_\infty^\pm$. Therefore we have the short exact sequence 
	$$0 \rightarrow \Sel_p^\pm(E/F_\infty) \rightarrow H^1(G_S(F_\infty),E_{p^\infty}) \rightarrow A_\infty^\pm \rightarrow 0.$$
	Taking $H$-cohomology and writing the long exact sequence corresponding to the above short exact sequence and using Lemma \ref{jishnu}, we get that $$H^1(H, A_\infty^\pm) \cong H^2(H, \Sel_p^\pm(E/F_\infty)).$$ But $H \cong \Z_p$ and so $H^2(H, \Sel_p^\pm(E/F_\infty))=0$ which implies that $H^1(H, A_\infty^\pm)=0.$ Hochschild-Serre spectral sequence gives
	$$ 0 \rightarrow H^1(\Gamma, J_v^\pm(F_\infty)^H) \rightarrow H^1(G,J_v^\pm(F_\infty)) \rightarrow H^1(H, J_v^\pm(F_\infty))^\Gamma \rightarrow H^2(\Gamma, J_v^\pm(F_\infty)^H)=0.$$

By the proof of Proposition 5.14 of \cite{LeiSujatha}, we know that $H^1(G, J_v^\pm(F_\infty)) =0$ for all $v \in S$. 

$\Big($Note that the proof of Proposition 5.14 of \cite{LeiSujatha} does not need $\Sel_p(E/F)$ to be finite. It just needs the fact that $H^2(F_S/F_\cyc, E_{p^\infty})=0$  which is true in our case of imaginary quadratic field by Theorem \ref{leopoldt}$\Big)$.

This implies $H^1(H, J_v^\pm(F_\infty))^\Gamma=0$. Now $J_v^\pm(F_\infty)$ is $p$-primary and $H^1(H, J_v^\pm(F_\infty))$ is also $p$-primary and $\Gamma$ is pro-$p$. Therefore, by Nakayama lemma, $H^1(H, J_v^\pm(F_\infty))^\Gamma=0$ implies 
\begin{equation}\label{2.3}
 H^1(H, J_v^\pm(F_\infty))=0.
\end{equation}
Let $$ 0 \rightarrow A_\infty^\pm \rightarrow J_v^\pm(F_\infty) \rightarrow B_\infty^\pm \rightarrow 0$$ be exact where $B_\infty^\pm =\Coker(\lambda_\infty^\pm)$. 
Therefore, by Lemma \ref{lemma2} and equation \eqref{2.3}  we get that $(B_\infty^\pm)^H = H^1(H, A_\infty^\pm) =0$. Again since $J_v^\pm(F_\infty)$ is $p$-primary, $B_\infty^\pm$ is $p$-primary and $H$ is pro-$p$. This again implies that $B_\infty^\pm=0$ which completes the proof. 
\end{proof}
\begin{lemma}\label{Cor5.12}
	We have $H^1(H,\Sel_p^\pm(E/F_\infty))=0$.
\end{lemma}
\begin{proof}
It follows from Lemma 	\ref{lemma2}, that the map $\lambda_\infty^{H,\pm}$ is surjective. By Proposition \ref{Prop:mithu} we know that the map $$\lambda_\infty^\pm: H^1(G_S(F_\infty),E_{p^\infty}) \rightarrow \oplus_{v \in S}J_v^\pm(E/F_\infty)$$ is surjective. Therefore the following sequence is short exact.
$$0 \rightarrow \Sel_p^\pm(E/F_\infty) \rightarrow  H^1(G_S(F_\infty), E_{p^\infty}) \xrightarrow{\lambda_{\infty}^{\pm}} \oplus_{v \in S}J_v^{\pm}(E/F_\infty)\rightarrow 0.$$
As $\lambda_\infty^{H,\pm}$ is also surjective, taking $H$-cohomology of the above exact sequence we get $H^1(H,\Sel_p^\pm(E/F_\infty))=0$ since  $H^1(H,H^1(G_S(F_\infty),E_{p^\infty}))=0$ by Lemma \ref{jishnu}.

\end{proof}
	
Let $W$ be a finitely generated $\Lambda(G)$-module. We write $W(p)$ for the submodule of all elements of $W$ killed by some power of $p$, $W_f = W/W(p)$. The homology groups $H_i(G,W)$ are the Pontryagin duals of the cohomology groups $H^i(G, \widehat{W})$. As shown in \cite{Howson_alone}, the modules $H^i(G,W)$ are finitely generated $\Z_p$-modules and so $H_i(G, W(p))$ are finite groups for all $i \geq 0$. Following Coates-Schneider-Sujatha \cite[Equation 15, page 196]{CoatesSchneiderSujatha_Links_between} and \cite[Corollary 1.7]{Howson_alone}, we define the $\mu$-invariant of $W$ using the Euler characteristic $$p^{\mu_G(W)} = \prod_{i \geq 0} \#H_i(G, W(p))^{(-1)^i} = \chi(G, W(p)).$$

By Lemma \ref{Cor5.12}, $H^1(H, \Sel^{\pm}(E/F_\infty)) =0$ which amounts to say that $H_1(H, X_{\infty}^\pm) =0$. Now let us consider the following exact sequence 
$ 0 \rightarrow X_\infty^\pm(p) \rightarrow X_\infty^{\pm} \rightarrow X_{\infty,f}^\pm \rightarrow 0$ and take $H$-homology. We obtain 
\begin{equation}\label{2.4}
0 \rightarrow H_1(H,  X_{\infty,f}^\pm) \rightarrow H_0(H,X_\infty^\pm(p)) \rightarrow H_0(H,X_\infty^\pm) \rightarrow H_0(H,X_{\infty,f}^\pm) \rightarrow 0.
\end{equation}
Recall $\big( X_{\infty}^\pm\big)_H \cong X_\cyc^\pm$ (cf. Lemma \ref{lem1}) which is finitely generated and torsion $\Lambda(\Gamma)$-module. This implies that $H_0(H,X_{\infty,f}^\pm) $ is also finitely generated and torsion (cf. equation \eqref{2.4}). Also $H_0(H,X_\infty^\pm(p))$ is finitely generated and killed by some power of $p$ and hence so is $H_1(H,  X_{\infty,f}^\pm)$ from \eqref{2.4}. In fact, $X_{\infty,f}^\pm = p^rX_{\infty}^\pm$ for some $r \geq 0$ and hence $X_{\infty,f}^\pm$  is a submodule of $X_\infty^\pm$. Now as $H$ has cohomological dimension $1$ and $H_1(H,X_\infty^\pm)=0$, by Lemma \ref{Cor5.12}, this implies that $H_1(H,  X_{\infty,f}^\pm)=0$. This proves the following Lemma.
\begin{lemma}\label{lemma6}
	The module $H_0(H, X_{\infty,f}^\pm)$ is finitely generated torsion $\Lambda(\Gamma)$-module and $H_1(H, X_{\infty,f}^\pm)=0$.
\end{lemma}
In the following,  we are going to relate the $\mu$-invariant of $X_\infty^\pm$ and $X_\cyc^\pm$ and their Euler characteristics. 
\begin{theorem}\label{mainTheoremSupersingular}
We have the following relation between $\mu$-invariants. $$\mu_G(X_\infty^\pm) = \mu_\Gamma(X_\cyc^\pm) - \mu_\Gamma(H_0(H, X_{\infty,f}^\pm)).$$
\end{theorem}
\begin{proof}
	By Lemma \ref{lemma6}, \eqref{2.4} is an exact sequence of finitely generated $\Lambda(\Gamma)$-torsion modules. Also $(X_\infty^\pm)_H \cong X_\cyc^\pm$ which implies $$\mu_\Gamma(H_0(H,X_\infty^\pm)) = \mu_\Gamma(X_\cyc^\pm).$$ By Hochschild-Serre, 
	$$0 \rightarrow H_0(\Gamma, H_1(H, X_\infty^\pm(p))) \rightarrow H_1(G,X_\infty^\pm(p)) \rightarrow H_1(\Gamma,H_0(H,X_\infty^\pm(p))) \rightarrow 0$$ is exact and $$H_2(G, X_\infty^\pm(p)) \cong H_{1}(\Gamma,H_1(H, X_\infty^\pm(p))).$$ As $G$  is a $p$-adic Lie group of dimension $2$, $H_i(G, X_\infty^\pm(p)) =0$ for $i > 2$. Therefore we have the equality of the following Euler characteristic formula
	$$\chi(G, X_\infty^\pm(p)) = \prod_{i=0}^1 \chi(\Gamma, H_i(H, X_\infty^\pm(p)))^{(-1)^i},$$
	and this   implies 
	$$\mu_G(X_\infty^\pm) = \prod_{i=0}^1(-1)^i\mu_\Gamma(H_i(H,X_\infty^\pm(p))).$$
	But by \eqref{2.4}, as $\mu$ is additive along finitely generated torsion $\Lambda(\Gamma)$-modules, $$\mu_\Gamma(H_1(H, X_{\infty,f}^\pm)) + \mu_\Gamma(X_\cyc^\pm) = \mu_\Gamma(H_0(H,X_\infty^\pm(p))) + \mu_\Gamma(H_0(H,X_{\infty,f}^\pm)).$$
	Therefore, $$\mu_G(X_\infty^\pm)= \mu_\Gamma(X_\cyc^\pm) +\delta +\epsilon,$$
	where $ \delta = \mu_\Gamma(H_1(H, X_{\infty,f}^\pm)) - \mu_\Gamma(H_0(H, X_{\infty,f}^\pm))$  and $\epsilon = -\mu_\Gamma(H_1(H, X_\infty^\pm(p)))$. Note that since $H_1(H, X_{\infty,f}^\pm)=0$  by Lemma \ref{lemma6}, $\delta = - \mu_\Gamma(H_0(H, X_{\infty,f}^\pm))$.
	
	In the following, we are going to show that $\epsilon=0$. By Lemma \ref{Cor5.12}, $H_i(H,X_\infty^\pm)=0$ for all $i \geq 1$. Taking $H$-homology of the following exact sequence 
	$$0 \rightarrow X_\infty^\pm(p) \rightarrow X_\infty^\pm \rightarrow X_{\infty,f}^\pm \rightarrow 0,$$
	we obtain $$H_2(H, X_{\infty,f}^\pm) \cong H_1(H, X_\infty^\pm(p)).$$ But $H_2(H,X_{\infty,f}^\pm) =0$ which gives $H_1(H,X_\infty^\pm(p)) =0$ and hence $\epsilon =0$. This shows that $ \mu_G(X_\infty^\pm) = \mu_\Gamma(X_\cyc^\pm) + \delta =  \mu_\Gamma(X_\cyc^\pm) - \mu_\Gamma(H_0(H, X_{\infty,f}^\pm))$ which completes the proof of our main Theorem.
\end{proof}
\begin{comment}
\begin{remark}
	Recall the general setting of Coates and Sujatha \cite{CoatesSchneiderSujatha_Links_between}, when $E$ has good ordinary reduction at all places of $F$ above $p$. Set $F_\infty =F(E_{p^\infty})$. By Weil pairing, $F_\infty$ contains the cyclotomic extension $F_\cyc$. Set $G=\Gal(F_\infty/F)$ and $H=\Gal(F_\infty/F_\cyc)$. Suppose $E$ is non CM, then a theorem of Serre gives us that $G$ is an open subgroup of $GL_2(\Z_p)$ for all $p$. That is $G$ is a $p$-adic Lie group of dimension $4$ and $H$ is a closed subgroup of $G$ of dimension $3$. Under this setting the $\mathfrak{M}_H(G)$ conjecture of Coates and Sujatha says that $X_{\infty}/X_{\infty}(p)$ is finitely generated over $\Lambda(H)$ (see Definition 5.1 and Conjecture 5.2 of \cite{Sujatha_Notes}). Here $X_{\infty}$ is the dual Selmer at the $F_\infty$ level. Under this conjecture, they show that $\mu_G(X_\infty)=\mu_\Gamma(X_\cyc)$ \cite[Prop. 5.3]{Sujatha_Notes}.  
\end{remark}
\end{comment}
This gives us the following Corollary.
\begin{corollary}\label{cor:mu}
Suppose $\mu_\Gamma(X_\cyc^\pm)=0$, then $\mu_G(X_\infty^\pm)=0$.
\end{corollary}
\begin{remark}\label{rem}
Note that in \cite{LoefflerZerbes1}, Loeffler and Zerbes constructed two variable $p$-adic $L$-functions $L_p^{+,+}$ and $L_p^{-,-}$ over $\Z_p^2$-extension such that the ideal generated by them is the characteristic ideal of $X_\infty^+$ and $X_\infty^-$ respectively. This extension of the classical Iwasawa main conjecture for the cyclotomic extension to the case of $\Z_p^2$-extension is conjectured due to Kim (see Conjecture $3.1$ in section 3 of \cite{Kim} and note that our $X_\infty^+$ is Kim's $X^{+/+}$ and our $X_\infty^-$ is Kim's $X^{-/-}$). Therefore, via Iwasawa main conjecture, if the $\mu_\Gamma$-invariants of the classical cyclotomic $p$-adic L-functions $L_p^\pm$ (cyclotomic Iwasawa main conjecture is $(L_p^\pm)= \mathrm{char}(X_\cyc^\pm)$) are zero, then by Corollary \ref{cor:mu}, the $\mu_G$-invariants of the Loeffler-Zerbes' two variable $p$-adic L-functions $L_p^{+,+}$ and $L_p^{-,-}$ should also be zero. (The Euler systems for the imaginary quadratic field case are given in \cite{LLZ1}).
\end{remark}

In the following, we recall the equivalent version of $\mathfrak{M}_H(G)$-conjecture for the supersingular case. This conjecture 
appeared in Section 5 of \cite{LeiZerbes1} and Conjecture 4.14 of \cite{LeiLim} in the supersingular case. See Proposition 4.14 of \cite{LeiLim} for a criterion when the $\mathfrak{M}_H(G)$-conjecture is valid. (Also see \cite{CFKSV} and \cite{CoatesSujatha1} for the case when $E$ has good ordinary reduction at $p$).
\begin{conjecture}[$\mathfrak{M}_H(G)$-conjecture]\label{ConjMH}
	Under the notations of Section \ref{main},  $X_{\infty,f}^\pm$ is finitely generated $\Lambda(H)$-module.
\end{conjecture}

	\begin{theorem}\label{thm:mufixed}
		The $\mathfrak{M}_H(G)$-conjecture is true if and only if $\mu_G(X_\infty^\pm) = \mu_\Gamma(X_\cyc^\pm)$.
	\end{theorem}
\begin{proof}
	Suppose $\mathfrak{M}_H(G)$-conjecture is true, then $H_0(H, X_{\infty,f}^\pm)$ is a finitely generated $\Z_p$-module and hence has $\mu_\Gamma$-invariant zero. Therefore, by Theorem  \ref{mainTheoremSupersingular}, we obtain $\mu_G(X_\infty^\pm) = \mu_\Gamma(X_\cyc^\pm)$. 
	
	Conversely, suppose $\mu_G(X_\infty^\pm) = \mu_\Gamma(X_\cyc^\pm)$, then by Theorem  \ref{mainTheoremSupersingular},  $\mu_\Gamma(H_0(H, X_{\infty,f}^\pm))=0$. From the exact sequence $(X_\infty^\pm)_H \rightarrow (X_{\infty,f}^\pm)_H \rightarrow 0$, since $(X_\infty^\pm)_H\cong X_\cyc^\pm$ is $\Lambda(\Gamma)$-torsion, we deduce that $(X_{\infty,f}^\pm)_H $ is also $\Lambda(\Gamma)$-torsion. Now since $\mu_\Gamma((X_{\infty,f}^\pm)_H )=0$, by the structure theorem, this implies that $(X_{\infty,f}^\pm)_H$ is a finitely generated $\Z_p$-module. As $H \cong \Z_p$, by Nakayama's Lemma, we have $X_{\infty,f}^\pm$ is finitely generated over $\Lambda(H)$.
\end{proof}

\section*{Acknowledgement}
The author would like to thank Antonio Lei and Sujatha Ramdorai for sending him their article \cite{LeiSujatha} and for their encourangement and support throughout this work. The author also thanks Meng Fai Lim for useful discussions and for reading an earlier draft of this article and suggesting several improvements. Thanks are also due to the anonymous referee for his/her careful reading and for pointing our that certain assumptions on an earlier draft of our article were unnecessary. The author also acknowledges the postdoc fellowship support from PIMS-CNRS. 

\bibliographystyle{alpha}
%\bibliography{IwasawaPG}
\bibliography{base_change_Iwahori,p_rational_field,Iwawasa_algebra_and_cornut,IwasawaPG}
\end{document}